\documentclass[11pt]{article}
\usepackage[utf8]{inputenc}
\usepackage{amsthm}
\usepackage{amsfonts}
\usepackage{amsmath}
\usepackage{colonequals}
\usepackage{epsfig}
\usepackage{url}
\newcommand{\GG}{{\cal G}}

\newtheorem{conjecture}{Conjecture}
\newtheorem{theorem}{Theorem}
\newtheorem{corollary}[theorem]{Corollary}
\newtheorem{lemma}[theorem]{Lemma}

\DeclareMathOperator{\wcol}{\mathrm{wcol}}
\DeclareMathOperator{\adm}{\mathrm{adm}}

\title{On weighted sublinear separators}
\author{Zdeněk Dvořák\thanks{Computer Science Institute, Charles University, Prague, Czech Republic. E-mail: {\tt rakdver@iuuk.mff.cuni.cz}.
Supported by the ERC-CZ project LL2005 (Algorithms and complexity within and beyond bounded expansion) of the Ministry of Education of Czech Republic.}}
\date{}

\begin{document}
\maketitle

\begin{abstract}
Consider a graph $G$ with an assignment of costs to vertices.  Even if $G$ and all its subgraphs admit balanced separators
of sublinear \emph{size}, $G$ may only admit a balanced separator of sublinear \emph{cost} after deleting a small set $Z$ of exceptional
vertices.  We improve the bound on $|Z|$ from $O(\log |V(G)|)$ to $O(\log\log\ldots\log |V(G)|)$, for any fixed number of
iterations of the logarithm.
\end{abstract}

A \emph{balanced separator} in a graph $G$ is a set $C\subseteq V(G)$ such that each component of $G-C$ has at most $\tfrac{2}{3}|V(G)|$
vertices.  A hereditary class of graphs $\GG$ has \emph{strongly sublinear separators} if there exist $c,\varepsilon>0$
such that every graph $G\in\GG$ has a balanced separator of size at most $c|V(G)|^{1-\varepsilon}$.
Famously, planar graphs~\cite{lt79} and more generally graphs without a fixed graph as a minor~\cite{alon1990separator},
as well as many geometrically defined graph classes~\cite{miller1997separators,miller1998geometric} have strongly sublinear
separators, and this property has important algorithmic applications~\cite{lt80}.

Building upon a result of Plotkin et al.~\cite{plotkin}, Dvořák and Norin~\cite{dvorak2016strongly} established a connection to
shallow minors.  For an integer $\ell\ge 0$, a \emph{depth-$\ell$ minor} of a graph $G$ is a graph obtained
from a subgraph of $G$ by contracting pairwise vertex-disjoint subgraphs, each of radius at most $\ell$.  The \emph{density}
of a graph $H$ is $|E(H)|/|V(H)|$, and we define $\nabla_\ell(G)$ as the maximum density of a depth-$\ell$ minor of $G$.
For a function $f:\mathbb{N}\to\mathbb{N}$, we say the \emph{expansion of $G$ is bounded by $f$} if $\nabla_\ell(G)\le f(\ell)$ for every
integer $\ell\ge 0$.  We say that a class of graphs $\GG$ has \emph{polynomial expansion} if there exists a polynomial $p$ bounding
the expansion of all graphs in $\GG$.
\begin{theorem}[Dvořák and Norin~\cite{dvorak2016strongly}]
A hereditary class $\GG$ of graphs has strongly sublinear separators if and only if it has polynomial expansion.
\end{theorem}

In this note, we will focus on the weighted version of the separators.  There are two senses in which one could interpret
this statement: The weights could affect the balance or the size of the separators.  While we primarily focus on the latter sense,
we will need to use the former sense in the course of our argument as well. To avoid confusion, we say that an assignment of non-negative
real numbers to vertices is a \emph{weight assignment} if the values are used to control the balance of the separators and
a \emph{cost assignment} if they are used to control the size of the separators.  Given a weight or cost assignment
$f:V(G)\to\mathbb{R}_0^+$ for a graph $G$ and a set $X\subseteq V(G)$, we define $f(X)=\sum_{x\in X} f(x)$,
and for a subgraph $H\subseteq G$, we define $f(H)=f(V(H))$.

Consider a cost assignment $\rho:V(G)\to\mathbb{R}_0^+$ for a graph $G$.  Even if $G$ belongs to a class with strongly sublinear separators,
it may not necessarily have a balanced separator of small cost.  For example, suppose $G$ is a star with center $v$ and leaves $v_1$, \ldots, $v_n$,
and let $\rho(v)=n/3$ and $\rho(v_i)=1$ for $i=1,\ldots,n$.  Any balanced separator $C$ must contain either $v$ or at least $n/3$ leaves,
implying $\rho(C)\ge n/3\ge \rho(G)/4$.  However, as in this example, it may be the case that a cheap separator exists after deleting a bounded
number of expensive vertices.  More precisely, for integers $t\ge 1$ and $q\ge 0$, we say a set $C\subseteq V(G)$ is
\emph{$(\rho/t)$-cheap with $q$ outliers} if there exists a set $C'\subseteq C$ of size at most $q$ such that
$\rho(C\setminus C')\le \rho(G)/t$.  In~\cite{onsubsep}, I conjectured that in graphs from classes with polynomial expansion,
it is possible to get a cheap balanced separator with only a bounded number of outliers.
\begin{conjecture}\label{conj-main}
For every polynomial $p$, there exists a function $q:\mathbb{N}\to\mathbb{N}$ such that the following holds.
Let $G$ be a graph with expansion bounded by $p$ and let $\rho:V(G)\to\mathbb{R}_0^+$ be a cost assignment for $G$.
For every integer $t\ge 1$, $G$ has a balanced separator which is $(\rho/t)$-cheap with $q(t)$ outliers.
\end{conjecture}
Let me remark that Conjecture~\ref{conj-main} is a weakening of my conjecture that graphs with polynomial expansion
are fractionally treewidth-fragile~\cite{twd}, and consequently the conclusion of Conjecture~\ref{conj-main} holds for
graphs from every class known to be fractionally treewidth-fragile.  This includes all proper minor-closed classes~\cite{devospart}
and all graphs with polynomial expansion and bounded maximum degree~\cite{twd}.  Moreover, in~\cite{twd}, I proved the following
weakening of this conjecture.

\begin{theorem}[Dvořák~\cite{onsubsep}]\label{thm-log}
For every polynomial $p:\mathbb{N}\to\mathbb{N}$, there exists a polynomial $q:\mathbb{N}\to\mathbb{N}$ such that the following holds.
Let $G$ be an $n$-vertex graph with expansion bounded by $p$ and let $\rho:V(G)\to\mathbb{R}_0^+$ be a cost assignment for $G$.
For every integer $t\ge 1$, $G$ has a balanced separator which is $(\rho/t)$-cheap with $q(t\log n)$ outliers.
\end{theorem}

As our main result, we make a further step towards Conjecture~\ref{conj-main} by improving the dependence on the number of vertices.
Let us define $\log^{(0)}(x)=x$ and $\log^{(i+1)} x=\log_2 \max\bigl(2,\log^{(i)}(x)\bigr)$ for $i\ge 0$.

\begin{theorem}\label{thm-iterlog}
For every polynomial $p:\mathbb{N}\to\mathbb{N}$ and integer $a\ge 0$, there exists a polynomial $q:\mathbb{N}\to\mathbb{N}$ such that the following holds.
Let $G$ be an $n$-vertex graph with expansion bounded by $p$ and let $\rho:V(G)\to\mathbb{R}_0^+$ be a cost assignment for $G$.
For every integer $t\ge 1$, $G$ has a balanced separator which is $(\rho/t)$-cheap with $\exp(q(t))\log^{(a)} n$ outliers.
\end{theorem}

Note that unlike Theorem~\ref{thm-log}, the bound in Theorem~\ref{thm-iterlog} is exponential in $t$, and thus
it beats Theorem~\ref{thm-log} only for $t=o(\log\log n)$.  Our approach gives a better bound in terms of weak coloring numbers.
Let $G$ be a graph and let $\prec$ be a linear ordering of its vertices.  For a vertex $v\in V(G)$ and an integer $r\ge 0$,
a vertex $u$ is \emph{$(\prec,r)$-reachable} from $v$ if there exists a path $P$ from $v$ to $u$ in $G$ of length at most $r$
such that $u\preceq x$ for every $x\in V(P)$ (and in particular, $u\preceq v$).  Let $L^\prec_r(v)$ denote the set of vertices
$(\prec,r)$-reachable from $v$, and let us define $\wcol^\prec_r(G)=\max\{|L^\prec_r(v)|:v\in V(G)\}$.
The \emph{weak $r$-coloring number} $\wcol_r(G)$ is the minimum of $\wcol^\prec_r(G)$ over all linear orderings $\prec$ of $V(G)$.
\begin{theorem}\label{thm-iterweak}
For every polynomial $p:\mathbb{N}\to\mathbb{N}$ and integer $a\ge 0$, there exists a polynomial $q:\mathbb{N}\to\mathbb{N}$ such that the following holds.
Let $G$ be an $n$-vertex graph with expansion bounded by $p$ and let $\rho:V(G)\to\mathbb{R}_0^+$ be a cost assignment for $G$.
For every integer $t\ge 1$, $G$ has a balanced separator which is $(\rho/t)$-cheap with $q(t)\wcol^5_{q(t)}(G)\log^{(a)} n$ outliers.
\end{theorem}
Note that weak coloring numbers and the expansion of a graph are linked, as described in the following lemma
which follows from known bounds (we give a more detailed argument in the Appendix).
\begin{lemma}\label{lemma-wcolw}
For every graph $G$ and integer $r\ge 1$, we have $\wcol_r(G)\le \Bigl(2r\nabla_r(G)\Bigr)^{3r}$ and $\nabla_r(G)\le \wcol_{4r}(G)$.
\end{lemma}
Hence, Theorem~\ref{thm-iterweak} implies Theorem~\ref{thm-iterlog}.
Let us remark that there exist graphs with polynomial expansion and superpolynomial weak coloring numbers~\cite{covcol}.
However, weak coloring numbers are known to be polynomial for a number of interesting graph classes~\cite{van2017generalised}.

Let us give two further interpretations of our main result.  For an integer $r\ge 1$, we say that $C\subseteq V(G)$ is
a \emph{balanced distance-$r$ separator} in a graph $G$ if $V(G-C)=A\cup B$ for disjoint sets $A$ and $B$ such that
$|A|,|B|\le \tfrac{2}{3}|V(G)|$ and the distance $d_G(A,B)$ between $A$ and $B$ in $G$ is greater than $r$.  Note that stars do not
have small balanced distance-$2$ separators.  However, this can again be worked around by deleting a few vertices first.
\begin{theorem}\label{thm-dist}
For every polynomial $p:\mathbb{N}\to\mathbb{N}$ and integers $a,r\ge 1$, there exists a polynomial $q':\mathbb{N}\to\mathbb{N}$ such that the following holds.
Let $G$ be an $n$-vertex graph with expansion bounded by $p$.  For every integer $t\ge 1$, there exists a set $Z\subseteq V(G)$ of size
at most $\exp(q'(t))\log^{(a)} n$ such that $G-Z$ has a balanced distance-$r$ separator of size at most $n/t$.
\end{theorem}
Theorem~\ref{thm-dist} is proved by applying Theorem~\ref{thm-iterweak} to a graph arising from the weak coloring number
at distance $r$; see the next section for more details.

A more direct application is to \emph{balanced edge separators}, sets of edges whose removal splits the graph into components
with at most $2/3$ of the vertices.  Again, as the example of stars shows, even very simple graphs do not admit
sublinear balanced edge separators.  However, applying Theorem~\ref{thm-iterweak} with the cost function $\rho(v)=\deg v$
and including in the edge separator all the edges incident with the non-outlier vertices,
we obtain the following result.
\begin{corollary}\label{cor-edge}
For every polynomial $p:\mathbb{N}\to\mathbb{N}$ and integer $a\ge 0$, there exists a polynomial $q:\mathbb{N}\to\mathbb{N}$ such that the following holds.
Let $G$ be an $n$-vertex graph with expansion bounded by $p$.
For every integer $t\ge 1$, there exists a set $Z\subseteq V(G)$ of size at most $q(t)\wcol^5_{q(t)}(G)\log^{(a)} n$
such that $G-Z$ has a balanced edge separator of size at most $2|E(G)|/t$.
\end{corollary}

We will prove Theorem~\ref{thm-iterweak} in the setting of separators whose balance is affected by weights of vertices
(this makes the result slightly more general, but additionally we need to use the weights in the proof).
For a weight assignment $w:V(G)\to\mathbb{R}_0^+$, a \emph{$w$-balanced separator} in a graph $G$ is a set $C\subseteq V(G)$ such
that each component $K$ of $G-C$ satisfies $w(K)\le \tfrac{2}{3}w(G)$.
We define $\eta(G,t)$ as the minimum integer $q$ such that for every weight assignment $w$ and cost assignment $\rho$,
there exists a $w$-balanced separator in $G$ which is $(\rho/t)$-cheap with $q$ outliers.

Let us remark that at least in this weighted setting, we cannot replace $q(t)$ by a constant independent of $t$ in the statement of Conjecture~\ref{conj-main}.
For $s\ge 3$, let $G_{s,n}$ be the graph obtained from the complete bipartite graph $K_{s,n}$ with parts $A$ and $B$ by subdividing each edge $s-2$ times.
Any depth-$r$ minor of $G_{s,n}$ is $2$-degenerate if $r<(s-2)/2$ and $s$-degenerate (and thus $(2r+2)$-degenerate)
if $r\ge (s-2)/2$, implying that $\nabla_r(G_{s,n})\le 2r+2$.
Let $w$ be the weight assignment for $G_{s,n}$ such that $w(y)=1$ for all $y\in B$ and $w(x)=0$ for all other vertices $x$.
Let $\rho$ be the cost assignment such that $\rho(x)=n$ for $x\in A$, $\rho(x)=s$ for $x\in B$, and $\rho(x)=1$ otherwise.
We have $\rho(G_{s,n})=s^2n$.  Consider an optimal $w$-balanced separator $C$ with $q<s$ outliers.
Since $C$ is $w$-balanced, for each $v\in A\setminus C$, $G_{s,n}$ contains at least $n/3$ paths from $v$ to $B$ that intersect $C$
(possibly in their last vertex).  This shows that $\rho(C\setminus A)\ge |A\setminus C|n/3$ (the vertices belonging to $B\cap C$ can
be counted for several vertices of $A$, but this is correct since their cost is $s$).
Hence, the cost of non-outlier vertices of $C$ is at least $\max(0,|C\cap A|-q)n+(s-|C\cap A|)n/3-\max(0,q-|C\cap A|)s$.
This expression is minimized for $n>3s$ when $|C\cap A|=q$, giving the lower bound $(s-q)n/3$ on the cost of non-outlier vertices.
If $(s-q)n/3\le \rho(G_{s,n})/t=s^2n/t$, then $q \ge s - 3s^2/t$.  For $s=t/6$, this implies $q\ge t/12$.
Hence, there exist graphs $G_t=G_{t/6,t/2+1}$ for $t=6,12,\ldots$ with $\nabla_r(G_t)\le 2r+2$ for each $r\ge 0$
and with $\eta(G,t)=\Omega(t)$.

\section{Weak coloring numbers, expansion, and distances}

Given a graph $G$, a linear ordering $\prec$ of its vertices, and an integer $m\ge 1$,
let $G^{[\prec,m]}$ denote the graph with vertex set $V(G)$ and the edges $uv$
for all $v\in V(G)$ and $u\in L^\prec_m(v)$.  We need a bound on the expansion of this graph
in terms of the expansion of $G$.

\begin{lemma}\label{lemma-exp}
Let $G$ be a graph, let $\prec$ be a linear ordering of its vertices, and let $m\ge 1$ and $r\ge 0$
be integers.  Then
$$\nabla_r\bigl(G^{[\prec,m]}\bigr)\le \bigl(\wcol_m^\prec(G)\bigr)^2(2m+1)^2(2r+1)^2\nabla_{(2m+1)r+m}(G)+\wcol_m^\prec(G).$$
In particular, if the expansion of $G$ is bounded by $p(r)=c(r+1)^k$ for some $c\ge 1$ and $k\ge 0$,
then the expansion of $G^{[\prec,m]}$ is bounded by $p'(r)=c'(r+1)^{k+2}$ for
$$c'=5\bigl(\wcol_m^\prec(G)\bigr)^2(2m+1)^{k+2}c.$$
\end{lemma}
\begin{proof}
For every $u\in V(G)$, let $R^\prec_m(u)$ be the subgraph of $G$ induced by $\{v\in V(G):u\in L^\prec_m(v)\}$.
Note that $v\in V(R^\prec_m(u))$ if and only if $G$ contains a path $P$ from $u$ to $v$ of length at most $r$
such that $u\preceq x$ for every $x\in V(P)$.  All vertices of such a path also belong to $V(R^\prec_m(u))$.
Consequently, $R^\prec_m(u)$ is connected and has radius at most $m$.  Moreover, every edge $uv\in E(G^{[\prec,m]})$
satisfies $R^\prec_m(u)\cap R^\prec_m(v)\neq \emptyset$, and every vertex $v\in V(G)$ belongs to $V(R^\prec_m(u))$
for $|L^\prec_m(v)|\le \wcol_m^\prec(G)$ vertices $u\in V(G)$.
The bound on $\nabla_r\bigl(G^{[\prec,m]}\bigr)$ follows by Lemma~3.10 of Har-Peled and Quanrud~\cite{har2015approximationjrnl}.
\end{proof}

For a set $X\subseteq V(G)$, let $G^{[\prec,m]/X}$ denote the induced subgraph of $G^{[\prec,m]}$ with vertex set
$\bigcup_{x\in X} L^\prec_m(x)$; note that $|V(G^{[\prec,m]/X})|\le \wcol^\prec_m(G)|X|$.
\begin{lemma}\label{lemma-sep}
Let $G$ be a graph, let $\prec$ be a linear ordering of its vertices, let $m\ge 1$ be an integer, and let $X$ be a set of vertices of $G$.
Let $C$ be a set of vertices of $G^{[\prec,m]/X}$.
If $x,y\in X$ belong to different components of $G^{[\prec,m]/X}-C$, then $d_{G-C}(x,y)>m$.
\end{lemma}
\begin{proof}
Suppose for a contradiction there exists a path $P$ of length at most $m$ from $x$ to $y$ in $G-C$,
and let $z=\min_\prec V(P)$.  If say $z=x$, then $P$ shows that $xy\in G^{[\prec,m]}$, contradicting the assumption
that $x$ and $y$ belong to different components of $G^{[\prec,m]/X}-C$.  Hence, we have $x\neq z\neq y$.
The subpaths of $P$ from $x$ and $y$ to $z$ show that $z\in L^\prec_m(x)\cap L^\prec_m(y)$, and thus
$xz,yz\in E(G^{[\prec,m]/X})$.  Since $x$ and $y$ belong to different components of $G^{[\prec,m]/X}-C$,
it follows that $z\in C$, which is a contradiction.
\end{proof}

We are now ready to derive the distance version of our main result.
\begin{proof}[Proof of Theorem~\ref{thm-dist}]
By Lemma~\ref{lemma-wcolw}, there exists an integer $c_r$ such that every graph $G$ with expansion bounded by $p$
satisfies $\wcol_r(G)\le c_r$.  By Lemma~\ref{lemma-exp}, there exists a polynomial $p'$ such that for every such graph $G$,
there exists a linear ordering $\prec$ of $V(G)$ such that the expansion of $G^{[\prec,r]}$ is bounded by $p'$.
Let $q$ be the polynomial from Theorem~\ref{thm-iterlog} for the given $a$ and with $p'$ playing the role of $p$.
Let us define $q'(t)=q(c_rt)+2$.

Consider any graph $G$ with expansion bounded by $p$ and let $\prec$ be a linear ordering of $V(G)$ such that the expansion of
$G^{[\prec,r]}$ is bounded by $p'$.  For $u\in V(G)$, let $R(u)=\{v\in V(G):u\in L^\prec_m(v)\}$ and let $\rho(u)=|R(u)|$.
By Theorem~\ref{thm-iterlog} applied with $t'=c_rt$ playing the role of $t$, there exist sets $C,Z_0\subseteq V(G)$ such
that $|Z_0|\le \exp(q(t'))\log^{(a)} n=\tfrac{1}{e^2}\exp(q'(t))\log^{(a)} n$, $\rho(C)\le \rho(G)/t'$, and $C\cup Z_0$ is a balanced separator in $G^{[\prec,r]}$.
In particular, we can divide the components of $G^{[\prec,r]}-(C\cup Z_0)$ into two parts $A'$ and $B$ of size at most $\tfrac{2}{3}n$.
Let $C'=\bigcup_{u\in C} R(u)$ and $A=A'\setminus C'$.  We have $|C'|\le \rho(C)\le\rho(G)/t'\le c_rn/t'=n/t$.
Let $Z_1$ be a set of $\min(2|Z_0|,\max(|A|,|B|)-\min(|A|,|B|))$ vertices chosen arbitrarily from the larger of the sets $A$ and $B$.
Letting $Z=Z_0\cup Z_1$, note that
$|Z|\le 3|Z_0|\le \exp(q'(t))\log^{(a)} n$, and that either $|A\setminus Z|=|B\setminus Z|\le \tfrac{1}{2}|V(G-Z)|$, or
$\max(|A\setminus Z|,|B\setminus Z|)\le \tfrac{2}{3}n - 2|Z_0|=\tfrac{2}{3}(n-|Z|)=\tfrac{2}{3}|V(G-Z)|$.

It remains to argue that $C'\setminus Z$ is a distance-$r$ separator in $G-Z$.  Indeed, suppose for a contradiction that $P$ is a path
of length at most $r$ in $G-Z\subseteq G-Z_0$ from a vertex $v_1\in A$ to a vertex $v_2\in B$, and let $u=\min_\prec V(P)$.  Note that
$v_1uv_2$ is a path from $v_1$ to $v_2$ in $G^{[\prec,r]}$, and since $u\not\in Z_0$ and $v_1$ and $v_2$ belong to different
components of $G^{[\prec,r]}-(C\cup Z_0)$, we have $u\in C$.  However, then $v_1\in R(u)\subseteq C'$.  This is a contradiction,
since $v_1\in A$.
\end{proof}

\section{Separators with few outliers}

For integers $t,b\ge 1$, let $\ell(t,b)$ be the minimum integer $\ell$ such that $(1+1/t)^\ell>b$.
Since $(1+1/t)^t\ge 2$, we have $\ell(t,b)\le t\lceil\log_2(b+1)\rceil$.
Consider a graph $G$ with a weight assignment $w$ and a cost assignment $\rho$.
For a set $X\subseteq V(G)$, let $N_G(X)$ denote the set of vertices in $V(G)\setminus X$ with a neighbor in $X$.  We say that $G$ is
a \emph{$(w,\rho,t)$-expander} if $\rho(G)\neq 0$ and for every $X\subseteq V(G)$ such that $w(X)\le w(G)/2$, we have $\rho(N_G(X))\ge \rho(X)/t$.
\begin{lemma}\label{lemma-dist}
Let $t,b_1,b_2\ge 1$ be integers.
Let $G$ be a graph and let $w$ and $\rho$ be a weight and a cost assignment for $G$
such that $G$ is a $(w,\rho,t)$-expander.  For any $X_1,X_2\subseteq V(G)$, if
$\rho(X_1)\ge \rho(G)/b_1$ and $\rho(X_2)\ge \rho(G)/b_2$, then $d_G(X_1,X_2)\le \ell(t,b_1)+\ell(t,b_2)$.
\end{lemma}
\begin{proof}
For $i\ge 0$ and $j\in\{1,2\}$, let $X_j^i$ denote the set of vertices of $G$ at distance at most $i$ from $X_j$.
Suppose for a contradiction that $d_G(X_1,X_2)>\ell(t,b_1)+\ell(t,b_2)$, and thus $X_1^{\ell(t,b_1)}\cap X_2^{\ell(t,b_2)}=\emptyset$
and $w\bigl(X_1^{\ell(t,b_1)}\bigr)+w\bigl(X_2^{\ell(t,b_2)}\bigr)\le w(G)$.
By symmetry, we can assume $w\bigl(X_1^{\ell(t,b_1)}\bigr)\le w(G)/2$.
Since $G$ is a $(w,\rho,t)$-expander, for $1\le i\le \ell(t,b_1)$, we have
$\rho(X_1^i)\ge (1+1/t)\rho(X_j^{i-1})$, and thus
$\rho(X_1^{\ell(t,b_1)})\ge (1+1/t)^{\ell(t,b_1)}\rho(X_1)\ge (1+1/t)^{\ell(t,b_1)}\rho(G)/b_1>\rho(G)$,
by the definition of $\ell(t,b_1)$ and the assumption that $\rho(G)\neq 0$.  This is a contradiction.
\end{proof}

For a graph $G$ and an integer $l\ge 1$, let $\omega_l(G)$ denote the largest clique that appears in $G$ as a depth-$l$ minor.
Clearly, $\omega_l(G)\le 2\nabla_l(G)+1$.
A \emph{tight depth-$l$ clique minor} in $G$ is a clique minor $K$ of depth $l$ where the vertex set of every bag
is covered by at most $\omega_l(G)-1$ paths of length at most $l$ with the same starting point.  Note this implies
$|V(K)|\le \omega_l(G)((\omega_l(G)-1)l+1)\le \omega^2_l(G)l$.
For a positive integer $n'$, let $\eta_{\le\!n'}(G,t)$ denote the maximum of $\eta(H,t)$ over all subgraphs $H$ of $G$ with at most $n'$ vertices.

Let us now prove a key result relating the number of outliers in a graph~$G$ to the number of outliers in a subgraph of $G$ with
polylogarithmic number of vertices.  Theorem~\ref{thm-iterweak} will follow by iterating this result.
The argument is based on the proof of Plotkin et al.~\cite{plotkin}, the novel idea being the way we break up the expensive
vertices in part (c).

\begin{theorem}\label{thm-main}
Let $G$ be a graph with $n$ vertices, let $\prec$ be a linear ordering of the vertices of $G$ and let $t\ge 1$ be an integer.
Let $l=2\ell(5t,5tn)$, $b=\omega^2_l(G)l$, $m=2\ell(5t,20t)$, $r=\lceil\log_{9/8} (20t/3)\rceil$, and
$n'=5bt\wcol_m^\prec(G)$.
Then
$$\eta(G,t)\le r\eta_{\le\!n'}(G^{[\prec,m]},5rt).$$
\end{theorem}
\begin{proof}
Let $w$ and $\rho$ be
a weight and a cost assignment for $G$.  We construct a
sequence of tuples $T_i=(A_i,B_i,C_i,D_i,K_i,r_i)$ (for $i=0,1,\ldots$), where $A_i$, $B_i$, $C_i$, $D_i$ and $K_i$
are pairwise disjoint subsets of $V(G)$ and $r_i\ge 0$ is an integer and for $i>0$ we have $A_{i-1}\cup K_{i-1}\subseteq A_i\cup K_i$,
$B_{i-1}\subseteq B_i$, $C_{i-1}\subseteq C_i$, $D_{i-1}\subseteq D_i$, and $r_{i-1}\le r_i$.
Let us define $R_i=G-(A_i\cup B_i\cup C_i\cup D_i\cup K_i)$,
and let $H_i$ be the graph with vertex set $\{v\in V(R_i):\rho(v)>\tfrac{1}{5bt}\rho(G)\}$
and with $uv\in E(H_i)$ if and only if $d_{R_i}(u,v)\le m$.

We maintain the following invariants for every $i\ge 0$:
\begin{itemize}
\item[(i)] $w(A_i)\le \tfrac{2}{3}w(G)$ and $N_G(A_i)\subseteq B_i\cup C_i\cup D_i\cup K_i$,
\item[(ii)] $\rho(A_i)\ge 5t\rho(B_i)$,
\item[(iii)] $\rho(C_i)\le \tfrac{1+r_i/r+\delta_i}{5t}\rho(G)$, where $\delta_i=0$ if $V(H_i)\neq\emptyset$ and $\delta_i=1$ otherwise,
\item[(iv)] $|D_i|\le r_i\eta_{\le\!n'}(G^{[\prec,m]},5rt)$,
\item[(v)] $K_i$ is the vertex set of a tight depth-$l$ clique minor such that
$\rho(v)\le \tfrac{1}{5bt}\rho(G)$ for every $v\in V(K_i)$, and
\item[(vi)] $r_i\le r$ and every component $M$ of $H_i$ satisfies $\rho(M)\le (8/9)^{r_i}\rho(G)$.
\end{itemize}
We let $A_0=B_0=D_0=K_0=\emptyset$, $r_0=0$, and $C_0=\{v\in V(G):\rho(v)<\tfrac{1}{5tn}\rho(G)\}$;
clearly, $\rho(C_0)\le \tfrac{1}{5t}\rho(G)$, and thus (iii) holds.  All the other invariants are trivially satisfied.

For $i\ge 0$, assuming we already determined
$(A_i,B_i,C_i,D_i,K_i)$, we proceed as follows.
\begin{itemize}
\item[(a)] If $w(R_i)\le \tfrac{2}{3}w(G)$, the construction stops.  By (i),
$B_i\cup C_i\cup D_i\cup K_i$ is a $w$-balanced separator in $G$.  By (ii),
we have $\rho(B_i)\le \tfrac{1}{5t}\rho(A_i)\le\tfrac{1}{5t}\rho(G)$.  Since $K_i$
is a vertex set of a tight depth-$l$ clique minor,
we have $|K_i|\le b$ and $\rho(K_i)\le \tfrac{1}{5t}\rho(G)$ by (v).
Together with (iii) and (vi), this implies $\rho(B_i\cup C_i\cup K_i)\le\rho(G)/t$.
By (iv) and (vi), it follows that the set $B_i\cup C_i\cup D_i\cup K_i$ is $(\rho/t)$-cheap with $r\eta_{\le\!n'}(G^{[\prec,m]},5rt)$ outliers, as required.

From now on, assume $w(R_i)>\tfrac{2}{3}w(G)$.  In particular, $V(R_i)\neq\emptyset$, and by the choice
of $C_0$, $\rho(R_i)\neq 0$.
\item[(b)] If $R_i$ is not a $(w,\rho,5t)$-expander, then let $Z_i\subseteq V(R_i)$ be such that
$w(Z_i)\le w(R_i)/2$ and $\rho(N_{R_i}(Z_i))<\tfrac{1}{5t}\rho(Z_i)$.  Let $A_{i+1}=A_i\cup Z_i$,
$B_{i+1}=B_i\cup N_{R_i}(Z_i)$, $C_{i+1}=C_i$, $D_{i+1}=D_i$, $K_{i+1}=K_i$ and $r_{i+1}=r_i$.
Note (ii) is satisfied by $T_{i+1}$, since it is satisfied by $T_i$ and $\rho(N_{R_i}(Z_i))<\tfrac{1}{5t}\rho(Z_i)$.
Moreover, $w(A_{i+1})\le w(A_i)+w(R_i)/2\le (w(G)-w(R_i))+w(R_i)/2=w(G)-w(R_i)/2<\tfrac{2}{3}w(G)$ since $w(R_i)>\tfrac{2}{3}w(G)$,
and thus $T_{i+1}$ satisfies (i).  All the other invariants are clearly preserved.

From now on, assume $R_i$ is a $(w,\rho,5t)$-expander.

\item[(c)] Let us now consider the case $V(H_i)\neq\emptyset$. If $\rho(H_i)\le \tfrac{1}{5t}\rho(G)$,
then let $T_{i+1}$ be obtained from $T_i$ by setting $C_{i+1}=C_i\cup V(H_i)$.
By (iii), we have $\rho(C_i)\le \tfrac{1+r_i/r}{5t}\rho(G)$, and thus $\rho(C_{i+1})\le \tfrac{2+r_{i+1}/r}{5t}\rho(G)$.
Moreover, $V(H_{i+1})=\emptyset$, and thus $T_{i+1}$ satisfies (iii).   All the other invariants are clearly preserved.
Hence, suppose that $\rho(H_i)>\tfrac{1}{5t}\rho(G)$.

Let $M_i$ be a component of $H_i$ with $\rho(M_i)$ maximum.  We claim that $\rho(M_i)\ge \tfrac{3}{4}\rho(H_i)$.
Indeed, suppose for a contradiction $\rho(M_i)<\tfrac{3}{4}\rho(H_i)$.  Then we can express $H_i$ as a disjoint
union of graphs $H'_i$ and $H''_i$ such that $\rho(H'_i),\rho(H''_i)>\rho(H_i)/4>\tfrac{1}{20t}\rho(G)$.
Since $R_i$ is a $(w,\rho,5t)$-expander, Lemma~\ref{lemma-dist} implies $d_{R_i}(V(H'_i),V(H''_i))\le m$.
But then two vertices of $R_i$ at distance at most $m$ from each other belong to different components of $H_i$,
contradicting the definition of $H_i$.

Therefore, we have $\rho(M_i)>\tfrac{3}{20t}\rho(G)$, and by (vi), $(8/9)^{r_i}>\tfrac{3}{20t}$,
implying $r_i<\log_{9/8}(20t/3)\le r$.

Let $G'_i=R_i^{[\prec,m]/V(H_i)}$, and let $w_i$ be the weight assignment defined by $w_i(v)=\rho(v)$ for $v\in V(H_i)$
and $w_i(v)=0$ otherwise.  By definition we have $G'_i\subseteq G^{[\prec,m]}$.
Moreover, we have $\rho(v)>\tfrac{1}{5bt}\rho(G)$ for each $v\in V(H_i)$,
and thus $|V(H_i)|\le 5bt$ and $V(G'_i)\le |V(H_i)|\wcol^\prec_m(G)\le 5bt\wcol^\prec_m(G)=n'$.
Therefore, $G'_i$ contains a $w_i$-balanced separator $X_i\cup Y_i$, where
$\rho(X_i)\le \tfrac{1}{5rt}\rho(G'_i)$ and $|Y_i|\le \eta_{\le\!n'}(G^{[\prec,m]},5rt)$.
Let $T_{i+1}$ be obtained from $T_i$ by setting $C_{i+1}=C_i\cup X_i$, $D_{i+1}=D_i\cup Y_i$, and $r_{i+1}=r_i+1$.
Clearly, all the invariants except possibly for (vi) are satisfied.

Let us now argue that (vi) holds.  Since $r_i<r$, we have $r_{i+1}\le r$.
Note that $R_{i+1}=R_i-(X_i\cup Y_i)$.
By Lemma~\ref{lemma-sep}, if vertices $u,v\in V(H_i)\setminus (X_i\cup Y_i)$
belong to different components of $G'_i-(X_i\cup Y_i)$, then $d_{R_{i+1}}(u,v)>m$,
and thus $uv\not\in E(H_{i+1})$.  Consequently, any component $M$ of $H_{i+1}$ is contained in a component of $G'_i-(X_i\cup Y_i)$.
Since $X_i\cup Y_i$ is a $w_i$-balanced separator in $G'_i$ and $\rho(M_i)\ge \tfrac{3}{4}\rho(H_i)$,
by the definition of $w_i$ we have $\rho(M)=w_i(M)\le \tfrac{2}{3}w_i(G'_i)=\tfrac{2}{3}\rho(H_i)\le \tfrac{8}{9}\rho(M_i)\le (8/9)^{r_i+1}\rho(G)$ by (vi).
This implies that $T_{i+1}$ satisfies (vi).

From now on, we assume $V(H_i)=\emptyset$.
\item[(d)] If there exists a bag $S$ of the clique minor $K_i$ with no neighbor in $R_i$,
let $T_{i+1}$ be obtained from $T_i$ by setting $A_{i+1}=A_i\cup S$ and $K_{i+1}=K_i\setminus S$.
We have $w(A_{i+1})\le w(G)-w(R_i)<w(G)/3$, and thus $T_{i+1}$ satisfies (i).
All other invariants are clearly satisfied as well.

\item[(e)] Let $S_1$, \ldots, $S_o$ be the bags of $K_i$.  By the previous paragraph, we can for $j=1,\ldots,o$
assume that $S_j$ has a neighbor $v_j\in V(R_i)$.
Fix any vertex $v\in V(R_i)$.  By the choice of $C_0$, we have $\rho(v),\rho(v_j)\ge \tfrac{1}{5tn}\rho(G)$.
Since $R_i$ is a $(w,\rho,5t)$-expander, Lemma~\ref{lemma-dist} implies that $d_{R_i}(v,v_j)\le l$.
Let $S$ be the union of at most $o$ paths of length at most $l$ joining $v$ to the vertices $v_1$, \ldots, $v_o$ in $R_i$.
Let $T_{i+1}$ be obtained from $T_i$ by setting $K_{i+1}=K_i\cup S$.  Clearly, (v) holds, since
$V(H_i)=\emptyset$, and thus $\rho(u)\le \tfrac{1}{5bt}\rho(G)$ for every $u\in V(R_i)$.
\end{itemize}
Note that after each of the operations (b), (c), (d), and (e),
the pair $(|A_{i+1}|+|C_{i+1}|+r_{i+1},|K_{i+1}|)$ is lexicographically strictly greater than $(|A_i|+|C_i|+r_i,|K_i|)$,
and since $|A_j|+|C_j|+r_j\le n+r$ and $|K_j|\le n$ for each $j$, the process necessarily stops after a finite number of steps.
\end{proof}

We now iterate Theorem~\ref{thm-main} $a$ times and use the trivial bound $\eta(H,t)\le |V(H)|$ at the end.
Let us remark that for $r_1,r_2\ge 1$, a graph $G$, and a linear ordering $\prec$ of the vertices of $G$,
if $H$ is a subgraph of $G^{[\prec,r_1]}$, then $H^{[\prec,r_2]}\subseteq G^{[\prec,r_1r_2]}$.
\begin{corollary}\label{cor-iter}
Let $G$ be a graph with $n$ vertices, let $\prec$ be a linear ordering of the vertices of $G$ and let $t\ge 1$ be an integer.
Let $t_0=t$, $n_0=n$, $m_0=1$ and for $i\ge 0$, let
\begin{itemize}
\item $r_i=\lceil\log_{9/8} (20t_i/3)\rceil$, $t_{i+1}=5r_it_i$,
\item $m_{i+1}=2\ell(5t_i,20t_i)m_i$,
\item $l_i=2\ell(5t_i,5t_in_i)$, $b_i=\omega^2_{l_i}(G^{[\prec,m_i]})l_i$, $n_{i+1}=5b_it_i\wcol_{m_{i+1}}^\prec(G)$.
\end{itemize}
For every $a\ge 0$, we have $\eta(G,t)\le n_a\prod_{i=0}^{a-1} r_i$.
\end{corollary}
Let us estimate the quantities from Corollary~\ref{cor-iter}.  Let $p(r)=c(r+1)^k$ be a polynomial bounding the expansion
of $G$.  We consider $p$ as well as the number of iterations $a$ to be fixed, and thus we hide multiplicative terms depending only
on them in the $O$-notation.  By Lemma~\ref{lemma-exp}, for every $l,m\ge 1$, we have
$$\omega_l(G^{[\prec,m]})=O(\nabla_l(G^{[\prec,m]}))=O\Bigl(\bigl(\wcol_m^\prec(G)\bigr)^2(ml)^{O(1)}\Bigr).$$
We have $r_i=O(\log t_i)$, and thus $t_i=O(t\log^i t)$ for $i\le a$.  Hence,
$m_{i+1}=O(m_it_i\log t_i)=O(m_it\log^{i+1} t)$, and thus $m_i=O(t^{O(1)})$ for $i\le a$.

Note that $n_{i+1}$ is a function of $l_i=O(t_i\log (t_in_i))$, and the logarithm diminishes
the effects of the earlier iterations.  Hence, we can with no great loss use the value of $m_i$ and $t_i$
from the last iteration in all previous iterations as well.  Moreover, we can choose the ordering $\prec$
so that $\wcol^\prec_{m_a}(G)=\wcol_{m_a}(G)$. Let $s=\wcol_{m_a}(G)$, so that $\omega_{l_i}\bigl(G^{[\prec,m_a]}\bigr)=O\bigl(s^2(m_al_i)^{O(1)}\bigr)$.  We have
\begin{align}
n_{i+1}&=O\Bigl(\omega^2_{l_i}\bigl(G^{[\prec,m_a]}\bigr)l_it_as\Bigr)\nonumber\\
&=O\bigl(s^5m_a^{O(1)}t_al^{O(1)}_i\bigr)\nonumber\\
&=O\Bigl(s^5\bigl(t\log \max(2,n_i)\bigr)^{O(1)}\Bigr).\label{eq-nit}
\end{align}
By Lemma~\ref{lemma-wcolw}, we have $\log s=O(m_a\log p(m_a))=O(m_a\log m_a)=O(t^{O(1)})$.
Hence, an inductive argument using (\ref{eq-nit}) implies $n_a=O\Bigl(s^5\bigl(t\log^{(a)}n\bigr)^{O(1)}\Bigr)$.

\begin{corollary}\label{cor-iterpoly}
For every polynomial $p:\mathbb{N}\to\mathbb{N}$ and integer $a\ge 1$, there exists a polynomial $q:\mathbb{N}\to\mathbb{N}$
such that the following holds.  For every graph $G$ and an integer $t\ge 1$,
if $G$ has expansion bounded by $p$, then
$$\eta(G,t)\le \wcol^5_{q(t)}(G) q\bigl(t\log^{(a)}n\bigr).$$
\end{corollary}

Theorem~\ref{thm-iterweak} follows from Corollary~\ref{cor-iterpoly}; we use the fact that $q(\log^{(a+1)}n)=O(\log^{(a)} n)$
to make the dependency on $\log^{(a)} n$ linear.

\bibliographystyle{siam}
\bibliography{../data}

\section*{Appendix}

Let $G$ be a graph and let $\prec$ be a linear ordering of its vertices.
For a vertex $v\in V(G)$ and an integer $r\ge 0$, let $\kappa^\prec_r(v)$ denote the maximum number of paths
of length at most $r$ in $G$ starting in $v$, pairwise disjoint except for $v$, and ending in $\{x\in V(G):x\prec v\}$.
Let $\adm^\prec_r(G)=\max\{\kappa^\prec_r(v):v\in V(G)\}$.  The \emph{$r$-admissibility} $\adm_r(G)$ of $G$
is the minimum of $\adm^\prec_r(G)$ over all linear orderings $\prec$ of $V(G)$.
By a detour via another notion (strong $r$-coloring number), it is easy to see that $\wcol^\prec_r(G)\le \Bigl(\adm^\prec_r(G)\Bigr)^{r^2}$,
see e.g.~\cite{apxdomin} for details.  However, a better bound follows by a direct argument.

\begin{lemma}\label{lemma-admw}
Let $G$ be a graph and let $\prec$ be a linear ordering of its vertices.  Then for every $r\ge 1$, we have
$$\wcol^\prec_r(G)\le \Bigl(r^2\adm_r^\prec(G)\Bigr)^r.$$
\end{lemma}
\begin{proof}
Let $\vec{H}$ be the auxiliary directed graph
with vertex set $V(G)\times \{0,\ldots, r\}$, edges $((u,i),(v,i+1))$ and $((v,i),(u,i+1))$ for each $uv\in E(G)$ and $0\le i\le r-1$,
and edges $((v,i),(v,i+1))$ for each $v\in V(G)$ and $0\le i\le r-1$.
For a vertex $(v,i)\in V(\vec{H})$, let $\pi(v,i)=v$, and for a path $Q$ in $\vec{H}$, let $m(Q)=\min_\prec \pi(V(Q))$.

Consider any vertex $v\in V(G)$ and let $\vec{T}$ be a minimal subgraph of $\vec{H}$ containing for each $u\in L^\prec_r(v)$
a path $P$ from $(v,0)$ to $(u,r)$ such that $m(P)=u$.
We claim that $\vec{T}$ has maximum indegree at most one.  Indeed, for any vertex $x\in V(\vec{T})$,
there must by the minimality of $\vec{T}$ exist a path $Q$ in $\vec{T}$ from $(v,0)$ to $x$.  Choose such a path $Q$
with $m(Q)$ maximum, and let $e$ be the last edge of $Q$.  If an edge $e'\neq e$ entered $x$, then $\vec{T}-e'$ would
contradict the minimality of $\vec{T}$, since in any path $P$ from $(v,0)$ in $\vec{T}$ containing the edge $e'$, we can replace the initial
segment by $Q$ without decreasing $m(P)$.

Therefore, $\vec{T}$ is an outbranching.  Furthermore, consider any vertex $x\in \vec{T}$, and let $P_1$, \ldots, $P_t$
be paths in $\vec{T}$ from $x$ to the leaves of $\vec{T}$ starting with pairwise different edges.  By the minimality of $\vec{T}$,
for $1\le i\le t$, the path $P_i$ has length at most $r$ and satisfies $m(P_i)\prec \pi(x)$.  Moreover, the paths
$P_1-x$, \ldots, $P_t-x$ are pairwise vertex-disjoint.  Let $F$ be an auxiliary graph with vertex set $\{1,\ldots,t\}$
and with $ij\in E(F)$ if and only if $\pi(V(P_i-x))\cap \pi(V(P_j-x))\neq\emptyset$.  Note that $F$ has maximum degree at most $r(r-1)\le r^2-1$,
and thus $F$ contains an independent set of size at least $t/r^2$.  On the other hand,
$\alpha(F)\le \kappa^\prec_r(\pi(x))\le\adm_r^\prec(G)$.  Hence, the maximum outdegree of $\vec{T}$ is at most
$r^2\adm_r^\prec(G)$.  Consequently, the number of leaves of $\vec{T}$, which equals $|L^\prec_r(v)|$, is
at most $\Bigl(r^2\adm_r^\prec(G)\Bigr)^r$.
\end{proof}

Grohe et al.~\cite{covcol} proved the following bound on the admissibility.
\begin{theorem}[Grohe et al.~\cite{covcol}, a consequence of Theorem~3.1]\label{thm-covcol}
For every graph $G$ and integer $r\ge 1$, we have $\adm_r(G)\le 6r\nabla_r^3(G)$.
\end{theorem}

The first inequality from Lemma~\ref{lemma-wcolw} is obtained by combining Theorem~\ref{thm-covcol} with Lemma~\ref{lemma-admw}.
For the second one, see~\cite[Observation~10]{espsublin}.

\end{document}